\documentclass[12pt]{article}
\usepackage{amsfonts, amssymb, amsmath, graphicx}

        \renewcommand{\phi}{\varphi}
        
        \renewcommand{\le}{\leqslant}        
        \renewcommand{\ge}{\geqslant}  
        
        \renewcommand{\Re}{\mathop{\rm Re\,}\nolimits}

        \newcommand{\C}{\mathbb C}
        \newcommand{\R}{\mathbb R}
        
        \newcommand{\D}{\mathbb D}
	
        \newcommand{\dd}{\partial}
	\newcommand{\db}{\bar\partial}

        \newcommand{\E}{\mathcal E}
        \newcommand{\wtL}{\widetilde L}

        \newtheorem{lemma}{Lemma}

        \newcommand{\be}[1]{\begin{equation}\label{#1}}
        \newcommand{\ee}{\end{equation}}
        \newcommand{\myref}[1]{$(\ref{#1})$}
        \newenvironment{proof}[1][.]{\medskip\par\noindent
                {\bf Proof{#1}\ }}{\hfill$\Box$\par\medskip\noindent}

        \newcounter{example}

\newcommand{\Erdos}{Erd\H{o}s} 
\author{Alexander Fryntov, Fedor Nazarov}
\title{New estimates for the length of the \Erdos-Herzog-Piranian lemniscate}
\date{May 8,  2008}

\begin{document}
\maketitle

\begin{abstract}
\noindent

Let $p(z)$ be a monic polynomial of a fixed degree $n\ge 1$. Consider the
lemniscate
        $$
        L_p := \{z:\, |p(z)|=1\}\,.
        $$ 
Let $|L_p|$ be the length of $L_p$.
\Erdos, Herzog, and Piranian conjectured that
        \begin{equation}\label{eq:1}
        |L_p|\le |L_{p_0}|=2n+O(1)
        \end{equation}
where $p_0(z)=z^n-1$. Despite the efforts of many people, the conjecture
still remains unresolved. The goal of the note is to present a 
new approach to the problem that allows one to show that $|L_p|\le
|L_{p_0}|$ when $p$ is sufficiently close to $p_0$ and
to prove the asymptotic estimate $|L_p|\le 2n+o(n)$ as
$n\to\infty$ for all monic polynomials $p$.
\end{abstract}

\subsection{Introduction}
In 1958 \Erdos, Herzog, and Piranian (\cite{ErHePi}, Problem 12) asked whether the polynomial $p_0(z)=z^n-1$ has the 
maximal length of the lemniscate $L=L_p=\{z\in\C:\,|p(z)|=1\}$ among all 
monic polynomials $p(z)=z^n+\dots$ of degree $n$.
After 50 years, this conjecture still remains unresolved. The first 
upper bound $|L|\le 4\pi n$ was obtained by Dolzhenko in 1960 in his
thesis \cite{Do1} and published in 1963 \cite{Do2}. Meanwhile, in 1961, Pommerenke \cite{Po} 
published a much worse estimate $74n^2$, which became known much wider than Dolzhenko's
result. Apparently unaware of Dolzhenko's work, Borwein \cite{Bo} published the estimate
$|L|\le 8e\pi n$ in 1995. The first real improvement came in 1999 when Eremenko
and Hayman \cite{ErHa} proved the conjecture for $n=2$, showed that all critical points
of the extremal polynomial must lie on the lemniscate, and obtained the estimate
$|L|\le 9.173n$ for all $n$. In 2007, this upper bound was superceded by the
estimate $|L|\le 2\pi n$ proved by Danchenko \cite{Da}, which remained the best published
upper bound by the moment of writing this article. Several more papers devoted
to or motivated by the lemniscate problem have been published (see \cite{KuTk}, \cite{Ba}, 
and \cite{PeWa}, for instance).

The goal of this note is to present a new approach to the problem
based on an explicit formula for the length. Unfortunately, we haven't been
able to get a full solution either but, at least, we managed to show that
$|L_p|$ attains a local maximum when $p=p_0$ and to obtain an upper bound
of the form $2n+o(n)$.

\subsection{Acknowledgements}
This work was done when the first named author was visiting University of Wisconsin,
Madison. Our special thanks go to Andreas Seeger, whose generous support made
this visit possible. We are also grateful to Alexandre Eremenko, Mikhail Sodin,
and Alexander Volberg for valuable discussions. The second named author
was partially supported by the NSF grant DMS0501067.

\subsection{Notation}
Throughout this paper, we denote by $p$ a monic polynomial, by $n$ its degree, by 
$\eta$ an arbitrary root of $p$, by $\zeta$ an arbitrary root of $p'$, and by
$\xi$ an arbitrary root of $pp'$. This notation will be used without any further 
comments, so, say, $\sum_\zeta$ will always mean the sum over all roots of $p'$ 
counted with their multiplicities, etc.

As usual, $\C$ and $\R$ stand for the sets of complex and real numbers respectively.
We will also denote by $D_r$ the disk $\{z\in \C:\,|z|<r\}$ and by $T_r$ its boundary
circumference $\{z\in \C:\,|z|=r\}$.   
We denote by $d(F,z)$ the distance from the point $z\in \C$ to the set $F\subset\C$.

If $f$ is a complex valued smooth function defined on an open set
$\Omega\subset\C$\,, we shall treat it as a function of complex 
variable $z=x+iy$ and use the complex notation for
the partial derivatives
	$$
	\dd f = f_z = \frac{1}{2}(f_x - i f_y)\,,
	\qquad
	\db f = f_{\bar z} = \frac{1}{2}(f_x + i f_y)\,,
	$$
and for the differential forms
	$$
	dz = dx + i dy\,,\qquad d\bar z = dx - i dy\,.
	$$
In this notation,
	$$
	d\bar z\wedge d z = 2idx\wedge dy
	$$
and
	$$
	df = f_{z}dz + f_{\bar z} d\bar z = \dd f dz + \db f d\bar z\,,
	$$  
All other notation conventions will be introduced at the moment of their 
first appearance.

\subsection{The Stokes formula}
Let $E$ be a bounded open subset of the complex plane $\C$\,.
Suppose that the boundary of $E$ consists of finitely many smooth Jordan
arcs. Let $\omega$ be a differential 1-form on $\C$ with locally bounded coefficients
that are smooth outside a finite set $K$
whose differential $d\omega$ (which makes sense everywhere in
$\C\setminus K$) has locally integrable coefficients with respect to the area measure. 
(we shall call such forms quasismooth). It is not hard to check that
the classical Stokes formula  
	\be{eq:stoks}
	\int_{\dd E} \omega = \iint_E d\omega\,
	\ee
remains valid for quasismooth 
differential 1-forms. 

Let now $p=p(z)$ be a monic polynomial of a fixed degree $n\ge 2$ (the case $n=1$ is rather trivial, so we will not consider it here).
With any such polynomial, we associate the sets  
	$$
	L=L_p =\{z\in\C:\, |p(z)|=1\}\,,
	\quad
	E = E_p=\{z\in \C:\,|p(z)|<1\}\,.
	$$
Let $K$ be a finite set containing the critical points of $p$, and let  
$s=s(z)$ be a smooth complex valued function defined on $\C\setminus K$. 
Suppose that $s$ coinsides with  
the outward unit normal vector to $\dd E$ on $\dd E\setminus K$\,.
Then
	\be{eq:length}
	|L| 
	= 
	\int_{L} \frac{dz}{is}=\frac 1i\int_{L} \bar s\, dz\,.
	\ee
If the differential 1-form $\omega = \bar s\,dz$
is quasismooth, then the Stokes formula 
yields
	$$
	i|L| = \iint_{E} d\omega\,. 
	$$
Evaluating $d\omega$ we have
	$$
	d\omega = (\dd\bar s\, dz + \db\bar s\,d\bar z)\,\wedge dz = 
	\db \bar s\,d\bar z\wedge dz = 2i\,\overline{\dd s}\,
	dx\wedge dy\,.
	$$
Therefore,
	$$
	|L| = 2\iint_{E} \overline{\dd s}\,dx\wedge dy\,.
	$$ 
Since the quantity on the left hand side is real, we conclude that also
	\be{eq:main}
	|L| = 2\iint_{E}
	\dd s\,dA =
	2\Re \iint_{E} \dd s\,dA\,,
	\ee
where $A$ is the plane area measure\,. 

The outward normal vector to $\dd E$ is the normalized gradient of 
the function $\log|p|$. Computing it, we get
	$$
  s = \frac{\overline\phi}{|\phi|}=\frac{|\phi|}{\phi}\,,
 	\quad\text{where}\quad \phi = \frac{p'}{p}\,
	$$
on $L_p$. We have a lot of freedom extending $s$ from $L_p$ to the entire
complex plane. The most obvious extension is given by the right hand side
of the last formula, which makes sense everywhere except the zeroes
and the critical points of $p$. This way we get a form $\omega=\bar s\,dz$
such that $\int_{L_p}\omega=i|L_p|$ and for every other piecewise 
smooth curve $\gamma$, we have $\left|\int_\gamma\omega\right|\le |\gamma|$. 
One more direct computation shows that
	$$
        2\dd s =  
        -\frac{|\phi|}{\phi}\,\frac{\phi'}{\phi}=
        \frac{|\phi|}{\phi}\Bigl(\sum_\eta\frac 1{z-\eta}-\sum_\zeta\frac 1{z-\zeta}\Bigr)\, 
	$$	
where the first sum is taken over all roots $\eta$ of $p$ and the second
sum is taken over all roots $\zeta$ of $p'$ (counted with multiplicities in both cases).	
It is not hard to see from here that $\omega$ is quasismooth and for every bounded open
set $G$, one has
$$
\Bigl|\iint_G d\omega\Bigr|\le \iint_G \Bigl(\sum_{\xi}\frac 1{|z-\xi|}\Bigr)\,dA(z)  
$$
where the sum is taken over all roots $\xi$ of $pp'$. Our first task will be to use this 
extension to show that $|L_p|\le |L_{p_0}|$ for all $p$ sufficiently close to $p_0(z)=z^n-1$.

\subsection{A simple lemniscate problem}
\begin{lemma}\label{le:simple}
Let $p(z)=z^n+a_2z^{n-2}+a_3z^{n-3}+\dots+a_n$ satisfy $a_n\in\mathbb R$, 
$\max_{2\le k\le n}|a_k|=1$. Let $\wtL_p=\{z\in\C:\Re p(z)=0\}$.
Then for every $r\ge 2$, one has 
$$
|\wtL_p\cap D_r|\le 2nr-c_n
$$
where $c_n>0$ depends on $n$ only.
\end{lemma}

\begin{proof}
Note first of all that the intersection $\wtL_p\cap T_\rho$ 
consists of at least $2n$ points for 
every $\rho\ge 2$. Indeed, $\Re z^n$ has $n$ positive maxima interlaced 
with $n$ negative minima
of size $\rho^n$ on $T_\rho$. All other terms together can contribute not more than 
$\rho^{n-2}+\rho^{n-3}+\dots+1<\rho^n$, so the signs of $\Re p(z)$ at those maxima 
and minima are the same as those of $\Re z^n$. Thus, we have at least $2n$ sign 
changes and, thereby, at least $2n$ roots for $\Re p(z)$ on $T_r$.

Now consider a big sphere $S$ of radius $R$ touching the complex plane at the 
origin (we think that the complex plane is horizontal and the sphere is above
the plane). Let $\tau$ be the central projection from the plane to the lower
hemisphere of the sphere $S$ with the center of the projection coinciding
with that of the sphere. Then, by the well-known Poincar\'e formula,
$$
|\tau(\wtL_p)|=\pi R \int_S\#(\E_\theta\cap \tau(\wtL_p))\,d\mu(\theta)
$$  
where $\mu$ is the surface measure on $S$ normalized by the condition $\mu(S)=1$ and 
$\E_\theta=\{x\in S:\langle x,\theta\rangle=0\}$ is the great circle
of $S$ orthogonal to $\theta$. 
Since every horizontal circumference on the lower hemisphere lying outside
$\tau(D_r)$ intersects $\tau(\wtL_p)$ at not fewer than $2n$ points, we conclude
that
$$
|\tau(\wtL_p\setminus D_r)|\ge \pi Rn-2nR\arctan\frac rR.
$$
Combining this inequality with the previous identity, we see that 
$$
|\tau(\wtL_p\cap D_r)|\le 2nR\arctan\frac rR-
\pi R\int_S [n-\#(\E_\theta\cap \tau(\wtL_p))]\,d\mu(\theta)\,.
$$
Since $\tau^{-1}(\E_\theta)$ is a line on
the complex plane $\C$, we conclude that for almost all $\theta$, one has
$$
\#(\E_\theta\cap \tau(\wtL_p))\le n\,,
$$
so the integrand is non-negative almost everywhere. 

Now let us pass to the limit as $R\to\infty$. The left hand side tends 
to $|\wtL_p\cap D_r|$. The first term on the right hand side tends to $2nr$.
Finally, the integral can be rewritten as  
$$
\int_\C [n-\#(\wtL_p\cap \Gamma_w)]
\left(1+\frac{|w|^2}{R^2}\right)^{-3/2}\,\frac{dA(w)}{2|w|}
$$
where $\Gamma_w$ is the line passing through $w\in\C$ in the direction 
perpendicular to $w$ (this is a simple exercise in the change of varible
theorem) and, thereby, it tends to
$$
\int_{\C}[n-\#(\wtL_p\cap \Gamma_w)]\,\frac{dA(w)}{2|w|}
$$
by the monotone convergence theorem.

Thus it remains to show that this integral is separated away from $0$ 
on the class of polynomials under consideration. We shall show that it is true
even for the smaller integral $\int_{D_1}$. 

The first observation is that for a fixed monic polynomial $p$ the number of the
intersections of $\wtL_p$ with $\Gamma_w$ is a stable quantity (meaning that it
doesn't change under slight perturbations of the coefficients of $p$) for almost
all $w$. This, together with the dominated convergence theorem, allows to conclude
that our integral is a continuous function of the coefficients $a_2,\dots,a_n$ 
running over a compact set. Thus, it will suffice to show that this integral never 
vanishes.

To this end, fix $p$ and let $j$ be the least index for which $a_j\ne 0$. Consider
the line $\Gamma(\alpha)=\{te^{i\alpha}:\,t\in\R\}$. We have
$$
\Re p(te^{i\alpha})=t^n\Re e^{in\alpha}+t^{n-j}\Re[a_je^{i(n-j)\alpha}]+\dots=q(t)\,.
$$  
Note now that $\Re[a_je^{i(n-j)\alpha}]$ preserves sign on intervals 
of length $\frac{\pi}{n-j}$ and that $\Re e^{in\alpha}$ changes sign on every
interval whose length exceeds $\frac{\pi}n$. Thus, we can find $\alpha$ such 
that $\Re[a_je^{i(n-j)\alpha}]$ and $\Re e^{in\alpha}$ are both non-zero and 
have the same sign. If $q(t)$ had $n$ real roots, every its derivative would 
have only real roots. But the $n-j$-th derivative is a polynomial of the form
$\kappa t^j+\lambda$ where $\kappa,\lambda\ne 0$ have the same sign and such
polynomial has at least one non-real root for all $j\ge 2$. Thus $q(t)$ must have 
non-real roots as well. Since the existence of non-real roots is a stable
property of polynomials, we conclude that for all $w$ sufficiently close
to $0$ whose argument is sufficiently close to $\alpha+\frac{\pi}2$, the 
line $\Gamma_w$ intersects $\wtL_p$ at strictly less than $n$ points, which
is more than enough to ensure the strict positivity of the integral in question. 
\end{proof}

A simple rescaling argument yields the following result. 
Let $p(z)=z^n+a_2z^{n-2}+a_3z^{n-3}+\dots+a_n$ satisfy $a_n\in\mathbb R$, 
$\max_{2\le k\le n}|a_k|^{1/k}=a>0$. Let $\wtL_p=\{z\in\C:\Re p(z)=0\}$.
Then for every $r\ge 2a$, one has 
$$
|\wtL_p\cap D_r|\le 2nr-c_na
$$ 
For what follows, it will be useful to note that under these assumptions,
the roots $\zeta$ of $p'$ lie in $D_{2a}$.

\subsection{The first variation estimate}

Suppose that $p$ is a monic polynomial that is close to $p_0$. Using an
appropriate shift of the argument, we can ensure that the coefficient 
at $z^{n-1}$ equals $0$. After that is done, we can make the free term 
real by using an appropriate rotation. Hence we can assume without 
loss of generality that $p=p_0+q$ where $q(z)=\sum_{k=2}^n a_k z^{n-k}$, the 
coefficients $a_k$ are small, and $a_n\in\R$.

Let 
$$
a=\max_k |a_k|^{1/k}\,.
$$
Choose $r\in(4a,\frac 14)$ and consider the parts of the domains $E_p$ and $E_{p_0}$
lying inside and outside $T_r$ separately. We shall start with the comparison
of $|L_p\setminus D_r|$ and $|L_{p_0}\setminus D_r|$.

Using the same differential 1-form $\omega$ 
as before, we get
$$
i|L_p\setminus D_r|+\int_{E_p\cap T_r}\omega=\iint_{E_p\setminus D_r}d\omega
$$
and
$$
\int_{L_{p_0}\setminus D_r}+\int_{E_{p_0}\cap T_r}\omega=\iint_{E_{p_0}\setminus D_r}d\omega
$$ 
Let $G$ be the symmetric difference of $E_p$ and $E_{p_0}$. Since 
$$
\Bigl|\int_\gamma\omega\Bigr|\le|\gamma| 
$$
for every piecewise smooth curve $\gamma$, we immediately conclude from the
above formulae that 
$$
|L_{p}\setminus D_r|\le |L_{p_0}\setminus D_r|+|G\cap T_r|+
\iint_{G\setminus D_r}\sum_\xi\frac{1}{|z-\xi|}\,dA(z)\,.
$$
Now let us notice that for every $\rho\in (4a,\frac 14)$, the circumference
of radius $\rho^n$ centered at $1$ is transversal to the unit circumference.
Since $z^n$ travels over that circumference at constant speed, simple geometric
considerations show that $|\,|p_0(z)|-1|\ge c\rho^{n-1}\min_j|z-z_j|$ 
when $z\in T_\rho$ where $z_j$ are the 2n solutions of the system 
$|p(z_0)=1|$, $|z|=\rho$, and $c>0$ is some absolute constant. 
Since $|q|\le 2a^2\rho^{n-2}$ on $T_\rho$, we immediately conclude that
$G\cap T_\rho$ is contained in the union of $2n$ arcs of length $Ca^2\rho^{-1}$, i.e.,
that 
\be{eq:cross_sect}
|G\cap T_\rho|\le Ca^2\rho^{-1}
\ee  
with some absolute positive constant $C$ depending on $n$ only. 
In particular, 
$$
|G\cap T_r|\le Ca^2r^{-1}\,.
$$
To estimate the double integral over $G\setminus D_r$, note that
since $p_0'$ doesn't vanish outside $D_{1/4}$, and since
$|q|\le 2^{n+1} a^2$ in $D_2$, we can use the regular perturbation theory 
to conclude that $G_1=G\setminus D_{1/4}$ is contained in the $Ca^2$-neighborhood
of $L_{p_0}$, so $A(G_1)\le Ca^2$. Also, when $a$ is small enough, the distance 
from $G_1$ to every root $\xi$ of $pp'$ is bounded from below by some constant
$c$ depending on $n$ only. Thus, 
$$
\iint_{G_1}\sum_\xi\frac{1}{|z-\xi|}\,dA(z)\le Ca^2\,.
$$
The other part $G_2=G\cap D_{1/4}\setminus D_r$ is more interesting.
If $\eta$ is a root of $p$, then we still have $d(G_2,\eta)\ge c$, so
the integral of $\frac 1{|z-\eta|}$ over $G_2$ does not exceed $CA(G_2)$.
In view of (\ref{eq:cross_sect}), we have
$$
A(G_2)\le Ca^2\int_r^{1/4}\frac{d\rho}{\rho}\le Ca^2\log\frac 1r\,.
$$ 
If $\zeta$ is a root of $p'$, then $|\zeta|\le 2a$ and, thereby,
$$
\frac{1}{|z-\zeta|}\le \frac 2{|z|}
$$
on $G_2$. Thus, the integral of $\frac 1{|z-\zeta|}$ over $G_2$ is
bounded by 
$$
Ca^2\int_r^\infty \frac{d\rho}{\rho^2}=Ca^2r^{-1}\,. 
$$
Bringing all these estimates together, we finally conclude that
$$
|L_{p}\setminus D_r|\le |L_{p_0}\setminus D_r|+Ca^2r^{-1}\,.
$$
Now it is time to compare $|L_p\cap D_r|$ with 
$|L_{p_0}\cap D_r|$. For the latter, we'll use the
trivial lower bound $|L_{p_0}\cap D_r|\ge 2nr$.
To estimate the former, consider the polynomial 
$f(z)=1+p(z)=z^n+q(z)$, the lemniscate 
$$
\wtL=\{z:\Re f=0\}\,.
$$
and the region $F_p=\{z:\, \Re f(z)>0\}$.
Obviously, $E_p\subset F_p$. The key observation is that the difference
$H=F_p\setminus E_p$ is rather small in the disk $D_r$. Indeed, since 
the inequality $|p|<1$ is equivalent to $2\Re f-|f|^2>0$ and $|f|\le 2r^n$
in $D_r$, we conclude that $H\cap D_r\subset\{z:\,|\Re f(z)|\le 2r^{2n}\}$.
Arguing as above, we see that 
$$
|L_p\cap D_r|\le |\wtL\cap D_r|+|H\cap T_r|+\iint_H\Bigl(\sum_\xi
\frac 1{|z-\xi|}\Bigr)\,dA(z)
$$
The length $|H\cap T_r|$ can be easily estimated from above by 
$Cr^{n+1}$. Indeed, the condition
$|\Re f|\le 2r^{2n}$ implies $|\Re z^n|\le 2(r^{2n}+a^2r^{n-2})\le \frac{r^n}2$.
Hence, the speed with which $\Re f$ changes when $z$ travels over $T_r$ with 
unit speed across $H$ is at least 
$$
\frac{nr^{n-1}}2-\sum_{k=2}^n (n-k)a^kr^{n-k-1}\ge \frac{n r^{n-1}}4\,, 
$$
from where the above estimate follows immediately.

To estimate the area integral, note that $\iint_H \frac{dA(z)}{|z-\xi|}$ has the 
geometric meaning of the average (over the angle the lines make with
the positive axis) length of the cross-sections of $H$ by the lines
passing through $\xi$. These cross-sections can be easily controlled by
the Remez theorem (see \cite{BoEr}, Theorem 5.1.1) that asserts that 
the measure of the set where the 
polynomial $at^n+\dots$ with real coefficients is less than $\varepsilon$
in absolute value is at most $4(\varepsilon/|a|)^{1/n}$. Applying this
estimate to the polynomial $\Re f(\xi+e^{i\alpha}t)$, we conclude that
for every $\xi$,
$$
\iint_{H\cap D_r}\frac{dA(z)}{|z-\xi|}\le 4\cdot 2^{1/n}r^2\int_0^\pi
\frac{d\alpha}{|\cos n\alpha|^{1/n}}=Cr^2\,.
$$ 
Finally, according to Lemma \ref{le:simple}, 
$$
|\wtL\cap D_r|\le 2nr-ca\,.
$$
Bringing all the estimates together, we obtain
$$
|L_p|\le |L_{p_0}|-ca+C(r^2+a^2r^{-1})
$$
and it remains to put $r=a^{2/3}$ and choose $a$ sufficiently small to get the desired result.

\subsection{The simplest upper bound for the lemniscate length}

We shall start with using the above extension to show that 
	$$
	|L|\le 2\pi(2n-1)\,.
	$$
We have
  $$
  |L|= - \iint_E \frac{|\phi|}{\phi}\,\frac{\phi'}{\phi}\,dA\,
	\le 
	\iint_E \left|\frac{\phi'}{\phi}\right|\,dA\,.
  $$
Recall now that 
  $$
  -\frac{\phi'}{\phi}=\sum_\eta \frac 1{z-\eta}-\sum_\zeta \frac 1{z-\zeta}
  $$
Thus    
  $$
  |L| 
	\le
  \sum_\xi \iint_E \frac{1}{|z-\xi|}\,dA\,,
   $$ 
Since the logarithmic capacity of $E$ is equal to $1$, P\'olya's
theorem (see \cite{Ra}, Theorem 5.3.5) implies that its area does not exceed $\pi$.  
Thus,
        $$
        |L|
	      \le 
        (2n-1)\iint_{\D} \frac{dA}{|z|} 
	= 
	2\pi(2n-1)\,.
        $$
Note, by the way, that if $p$ has multiple roots, then every root of $p$ of
multiplicity $m$ appears $m$ times in $\sum_\eta$ and $m-1$ times in $\sum_\zeta$.
If we take this cancellation into account, then the last estimate can be
improved to $2\pi(2k-1)$ where $k$ is the number of distinct roots of $p$.

\subsection{An improved upper bound.}

To obtain a better estimate, we consider a different extension of $s$,
which differs from the one we considered before by the factor $|p|$
(this factor is identically $1$ on $L$). So, we define
	$$
	s = |p|\frac{|\phi|}{\phi}=\frac{p|p'|}{p'} = \frac{|p'|}{\phi}\,.
	$$
In this case 
	$$
	2\dd s = \frac{p|p'|}{p'}
	\Big(
        \frac{2p'}{p} - \frac{p''}{p'}
	\Big)
	=
	2|p'|-\frac{|p\phi|}{\phi}\psi
	=|p'|-\frac{|p\phi|}{\phi}\frac{\phi'}{\phi}
	\,,
	$$
where
	$$
	\psi := \frac{p''}{p'} = \sum_{\zeta} \frac{1}{z-\zeta}\,.
  $$
The lemniscate length can now be evaluated as follows
	\be{eq:main_two}
	|L| = 2\iint_{E}
	|p'|\, dA - 
	\iint_{E} \frac{|p\phi|}{\phi}\, \psi\, dA\,.
	\ee
Since $p$ maps $E$ onto the unit disk covering each
point of the unit disk $n$ times, we have
	$$
	\iint_{E} |p'|^2\, dA = \pi n\,.
	$$
The Cauchy inequality yields
	$$
	\iint_{E} |p'|\, dA \le \pi \sqrt{n}\,.
	$$
On the other hand,
	$$
	\Bigl|\iint_{E} \frac{|p\phi|}{\phi}\, \psi\, dA \Bigr| 
	 \le \iint_E
	|\psi|
	\,dA 
	\le 
	\sum_{\zeta} \iint_E \frac{1}{|z-\zeta|} 
	\le
	(n-1)\iint_\D\frac{1}{|z|}\, dA = 2\pi(n-1)\,.
	$$
Combining these estimates, we obtain the upper bound
	$$
	|L| \le 2\pi(n-1+\sqrt{n})\,,
	$$
which is only marginally worse than Danchenko's estimate $2\pi n$.

\subsection{Asymptotic estimate}
In this section we obtain an asymptotic 
estimate of the lemniscate length of a monic polynomial
$p$ of degree $n$ as $n$ approaches infinity.

Using the same extension as in the previous section, we get 
the inequality
        $$
        |L|\le \pi\sqrt n + J\,,
        $$
where
        $$
        J 
        =
        - \Re \iint_E
        \frac{|\phi p|}{\phi}\,
        \frac{\phi'}{\phi}\, dA\,.
        $$
In this section we will estimate the integral $J$ more accurately to
get the best possible asymptotic estimate for $|L|$\,.

Take $\delta\in(0,\frac14)$. Let $E_\delta\subset E$ be the set 
of all points $z\in E$ satisfying 
        \be{eq:e_delta_one}
        \Big|
        \frac{\phi'}{\phi}
        \Big| \ge
        2\delta n + \sum_\xi
        \frac{4\delta}{|z-\xi|}
        \ee
and
        \be{eq:e_delta_two}
        |z-\xi| \ge \frac{2\delta}{\sqrt{n}}\qquad\text{for all }\xi\,.         
        \ee
We shall start with estimating the integral over $E\setminus E_\delta$.
The set $E\setminus E_\delta$ is the union of the set
$B_1$ where condition \myref{eq:e_delta_one} is broken and the set
$B_2$ where condition \myref{eq:e_delta_two} is broken\,. Since
the area of the latter set is at most $4\pi\delta^2$\,,
we have 
\begin{align}
\iint_{B_1}\left|\frac{\phi'}{\phi}\right|\,dA &\le 2\delta n A(B_1)+\sum_\xi 4\delta\iint_{B_1}
\frac{dA(z)}{|z-\xi|}\le 2\pi n\delta+8\pi(2n-1)\delta
\\
\iint_{B_2}\left|\frac{\phi'}{\phi}\right|\,dA &\le \sum_\xi 4\delta\iint_{B_2}
\frac{dA(z)}{|z-\xi|}\le 4\pi(2n-1)\delta\,, 
\end{align}
so
  $$
	\iint_{E\setminus E_\delta}
        \Big|
        \frac{\phi'}{\phi}
        \Big|\, dA \le 2\pi n\delta + 8\pi(2n-1)\delta
        + 4\pi\delta (2n-1)
        \le 26\pi\delta n
        $$
and hence,
        $$
        J \le 26\pi\delta n+\iint_{E_\delta}|p'|\,dA+J_\delta
        $$
where   
        $$     
        J_\delta=
        -\Re \iint_{E_\delta}
        \frac{|\phi p|}{\phi}
        \,\psi\,dA\,.
        $$
Using the estimate 
	$$
	\iint_{E} |p'|\, dA \le \pi\sqrt{n}\,,
	$$
once more,
we see that it is enough to estimate $J_\delta$.
	
We rewrite $J_\delta$ 
using the explicit expression for  $\psi$ now:
	\be{eq:explicit}
        J_\delta = \sum_\zeta
        \Re \iint_{\delta} \frac{|p\phi|}{\phi}
        \,\frac{1}{\zeta - z}\,dA\,.
	\ee	
	
Now partition the complex plane into equal squares with side
length 
	$$
	\ell = \frac{\delta^2}{\sqrt{n}}\,.
	$$
and define the set $F$ as the minimal union of these squares that
contains $E_\delta$, that is, $F$ includes all the squares 
that intersect $E_\delta$\,. For each square $Q$ in $F$, let 
$\widetilde Q$ be the twice larger square with the same center.
Finally, let 
$$
\widetilde F=\bigcup_{Q:Q\subset F}\widetilde Q\,.
$$
Note that
	\be{eq:distance}
	d(\widetilde F, \xi) > 
	\frac{2\delta}{\sqrt n}-\frac{4\delta^2}{\sqrt n}>
	\frac{\delta}{\sqrt{n}}\,,
	\ee
for every root $\xi$ of $pp'$. 

Since for every function $\Phi$
        $$
        \iint_{E_\delta} 
	      \Re (|p|\,\Phi)\, dA \le
        \iint_{E_\delta} \Re_+\Phi\, dA \le
        \iint_{F} \Re_+ \Phi\, dA\,,
        $$
we obtain 
        \be{eq:estim_on_fi_delta}
        J_\delta \le 
        \sum_\zeta  
        \iint_{ F} 
        \Re_+\Big(\frac{|\phi|}{\phi}\,
        \frac{1}{\zeta - z}
        \Big)
        \,dA(z)\,.
        \ee        
Let now
	\be{eq:square_integral}
	I_\zeta(Q) : = 
	\iint_{Q} \Re_+ \frac{|\phi|}{(\zeta-z)\phi}\,
	\, dA(z)\,.
	\ee
where $Q$ is a square of the set $F$ and
$\zeta$ is a root of $p'$\,.
Let $w$ be any point in $Q$. Then
        $$
        I_\zeta(Q) 
        \le
        \iint_Q \Re_+ 
        \Big(
        \frac{|\phi|}{(\zeta - z)\phi}
        -
        \frac{|\phi|}{(\zeta - w)\phi}
        \Big)
        \,dA(z) 
        + 
        \frac{1}{|\zeta-w|}
        \iint_{Q}
        \Re_+\theta\, dA\,,
        $$
where
	$$
	\theta = \frac{|(\zeta-w)|}
	{(\zeta-w)}\,\frac{|\phi|}{\phi}\,.
	$$
Denote the first term by $I_\zeta^{(1)}(Q)$ and the
second one by $I_\zeta^{(2)}(Q)$\,.	

We have
        $$
        I_\zeta^{(1)}(Q) \le
        \iint_{Q} \frac{|w - z|}{|\zeta-z|\,|\zeta -w|}
        \, dA(z)\,
        \le
        \frac{2\ell}{|\zeta-w|}
        \iint_{Q}
        \frac{1}{|z-\zeta|}\, dA\,.
        $$

Since $\ell=\delta^2/\sqrt{n}$ and $|w-\zeta|\ge \delta/\sqrt{n}$
(see \myref{eq:distance}), we have
	\be{eq:i_one}
	I_\zeta^{(1)}(Q) \le
	2\delta \iint_{Q}\frac{1}{|\zeta-z|}\, dA(z)\,.
	\ee
To estimate $I_\zeta^{(2)}(Q)$, note that
	\be{eq:theta_series}
	\Re_+\theta = \frac{1}{\pi}
	+ \sum_{k\ne 0} \,a_k \,\theta^k\,.
	\ee
where $a_k$ are the Fourier coefficients of the function $\Re_+ z$
on the unit circumference.
To estimate the sum over $k\ne 0$, we need
\begin{lemma}\label{le:oscillation}
Let $Q$ be a square with size length $1$ and
let $u$ be a real harmonic function in the twice
larger square
$\widetilde Q$ with the same center 
such that
        $$
        |\db u| > R\,,
        $$
everywhere in $\widetilde Q$. Then
        $$
        \Big|
        \iint_{Q} e^{iu}\, dA
        \Big| \le \frac{4}{R}\,.
        $$
\end{lemma}
\begin{proof}
The function $f = 1/\dd u$ is analytic in the square
$\widetilde Q$
and $|f| < 1/R$. By Cauchy theorem $|f'|\le 2/R$ in $Q$\,.
Since $u$ is a real function, we have $\db u = 1/\bar f$, and
the  Stokes formula yields
        $$
        \int_{\dd Q} \bar f e^{iu} dz = 
        \iint_{Q} 
        \db (\bar fe^{iu})\, d\bar z\wedge dz =
        \iint_{Q} 
        (\bar f' e^{iu} + i e^{iu})\, d\bar z \wedge dz\,.
        $$
Thus
        $$
        \Bigl|
        \iint_{Q} e^{iu}\, dA\,
        \Bigr| 
        \frac 12 \left[
        \int_{\partial Q}|\bar f|\,|dz|+2\iint_Q|\bar f'|\,dA
        \right]
        \le 
        \frac4R\,.
        $$
We are done.
\end{proof}
A simple scaling argument shows that if the side length of $Q$ is not $1$ but
$\ell$, then the estimate changes to
        $$
        \Bigl|
        \iint_{Q} e^{iu}\, dA\,
        \Bigr| 
        \le 
        \frac{4}{R\ell}A(Q)\,.
        $$
We apply this lemma to estimate
        $$
        \iint_{Q}\, \theta^ k \, dA\,.
        $$
Take $w\in Q$ to satisfy
	$$
	\frac{\ell^2}{|\zeta-w|} =
	\frac{A(Q)}{|\zeta-w|}\,=
	\iint_{Q}\frac{1}{|\zeta-z|}\, dA\,.
	$$
By Lemma \ref{le:oscillation}
        $$
        \frac{1}{|\zeta-w|}
        \iint_{Q} \theta^k\, dA \le
        \frac{4}{R\ell|k|}\,
	      \iint_{Q}
	      \frac{1}{|\zeta-z|}\, dA\,,
        $$
where	
        $$
        R = 
        \inf_{z\in\widetilde Q} |\dd u|\,,
        $$
It remains to estimate $R$ from below.
We have
	$$
	|\dd u| = 
	\frac{1}{2}\,.
	\left|
	\frac{\phi'}{\phi}
	\right|
	$$
Take $z_0\in E_\delta\cap Q$. According to 
\myref{eq:e_delta_one} we have
        $$
        |\dd u(z_0)| \ge \delta n + 
        \sum_{\xi}\frac{2\delta}{|\xi - z_0|}\,.
        $$
Since $d(\widetilde Q,\zeta)\ge \delta/\sqrt{n}$,
$|z-z_0|\le 4\ell$, and $\delta\le \frac14$.
we get
	$$
	\frac{1}{2} \le
	\frac{|\xi -z_0|}{|\xi - z|} \le 2\,.
	$$
Note that 
$$
\left(\frac{\phi'}{\phi}\right)'=\sum_\eta \frac 1{(z-\eta)^2}-\sum_{\zeta}\frac{1}{(z-\zeta)^2}\,.
$$ 
Hence
	$$
	\left|\frac 12\left(\frac{\phi'}{\phi}\right)'\right|
	 \le
	\sum_\xi \frac{1}{2|z-\xi|^2} \le
        \frac{\sqrt{n}}{\delta}
        \sum_\xi \frac{1}{2|z-\xi|} 
	 \le
        \frac{\sqrt{n}}{\delta}
        \sum_\xi \frac{1}{|z_0-\xi|} 
        \,.
        $$
Finally, for $z\in\widetilde Q$ we have
        $$
        |\dd u(z)| \ge |\dd u(z_0)| - 2\ell\,
	\max_{\widetilde Q} \left|\frac 12\left(\frac{\phi'}{\phi}\right)'\right|
        \ge \delta n\,.
        $$
and thereby, in view of \myref{eq:theta_series}, 
	\be{eq:i_two}
	I_\zeta^{(2)}(Q) \le \frac{4a}{\delta^3\sqrt{n}}
	\iint_Q\frac{1}{|\zeta - z|}\, dA
        + \frac{1}{\pi}
	\iint_{Q}\frac{1}{|\zeta-z|}\, dA\,,
	\ee
where 
	$$
	a = \sum_{k\ne 0}\frac{|a_k|}{|k|}\,.	
	$$
Summing the inequalities \myref{eq:i_two} and
\myref{eq:i_one} over all squares $Q\in\widetilde F$
and then over $\zeta$ we get
	\be{eq:prefinal}
	J_\delta \le  
	\left(\frac{1}\pi + 2\delta+ \frac{4a}{\delta^3\sqrt n}\right)
        \iint_{F}
	\sum_\zeta\frac{1}{|\zeta-z|}\, dA\,.	
	\ee

To finish the estimate we need to estimate
the logarithmic capacity of the set $F$\,.
To this end, we estimate the logarithmic derivative $\phi$ of $p$ first. We have
	$$
	|\phi| \le 
	\sum_{\eta}
	\frac{1}{|\eta - z|}
	\le
	\frac{n\sqrt{n}}\delta
	$$
on $F$ (see \myref{eq:distance}).
Since $d(E,z)\le 2\ell$ for every point $z\in F$, we have
	$$
	\log|p|\le 2\ell\max_{F}|\phi| \le
	\frac{2\ell n\sqrt{n}}{\delta} = 2\delta n\,.
	$$
Thus,
	$$
	|p|\le e^{2\delta n}\,.
	$$
The last inequality means that the logarithmic capacity of
$F$ is at most $e^{2\delta}$ and, thereby, its area is at most 
$\pi e^{4\delta}$\,. Thus,
	$$
	|L| \le 
	26\pi\delta n + 2\pi\sqrt n +
	e^{2\delta}\left(\frac{1}\pi + 2\delta+ \frac{4a}{\delta^3\sqrt n}\right) 2\pi(n-1)
	$$
for every $\delta>0$\,. Now it is clear that the optimal choice of $\delta$ is about 
$n^{-1/8}$, which results in the estimate 
$$
|L|\le 2n+O(n^{7/8})\,.
$$
We have no doubt that the power $7/8$ can be substantially improved though to 
bring it below $1/2$ seems quite a challenging problem.


\begin{thebibliography}{10}


\bibitem{ErHePi} P. \Erdos, F. Herzog, and G. Piranian, \emph{Metric properties of polynomials}, J. Anal. Math. 6 (1958), 125–148. Also available at {\tt http://www.math-inst.hu/\raisebox{-6pt}{$\widetilde{\phantom{x}}$}p\_erdos/1958-05.pdf}

\bibitem{Do1} E.P. Dolzhenko, \emph{Differentiability properties of 
functions and certain questions of
approximation theory}, PhD (Candidate of science degree) Thesis, Moscow State University, Moscow 1960. (Russian)

\bibitem{Do2} E.P. Dolzhenko, \emph{Some metric properties of algebraic hypersurfaces}, 
Izv. Akad. Nauk SSSR Ser. Mat. 27:2 (1963), 241–252. (Russian)

\bibitem{Po} Ch. Pommerenke, \emph{On metric properties of complex polynomials}, 
Michigan Math. J. 8:2 (1961), 97–115.

\bibitem{Bo} P. Borwein, \emph{The arc length of the lemniscate $|p(z)| = 1$}, 
Proc. Amer. Math. Soc. 123:3 (1995), 797–799.

\bibitem{ErHa} A. Eremenko and W. Hayman, \emph{On the length of lemniscates}, 
Michigan Math. J. 46:2 (1999), 409–415. Also available at {\tt http://arxiv.org/abs/0805.2295}

\bibitem{Da} Danchenko, V. I. \emph{The lengths of lemniscates. Variations of rational functions}. (Russian) Mat. Sb.  198  (2007),  no. 8, 51--58;  translation in  Sb. Math.  198  (2007),  
no. 7-8, 1111--1117 

\bibitem{KuTk} Kuznetsova, O. S.; Tkachev, V. G. \emph{Length functions of lemniscates}.  
Manuscripta Math.  112  (2003),  no. 4, 519--538. Also available at 
{\tt http://arxiv.org/abs/math/0306327}

\bibitem{Ba} Barsegian, G. A. \emph{Gamma-lines of polynomials and a problem by Erdös-Herzog-Piranian}. Topics in analysis and its applications,  119--122, NATO Sci. Ser. II Math. Phys. Chem., 147, 
Kluwer Acad. Publ., Dordrecht, 2004.

\bibitem{PeWa} C. Wang and L. Peng, \emph{The arc length of the lemniscate $|w^n + c| = 1$}, 
Rocky Mountain J.Math. 36:1 (2006), 337–347.

\bibitem{Ra} T. Ransford, \emph{Potential theory in the complex plane}.
London Mathematical Society Student Texts, 28. Cambridge University Press,
Cambridge, 1995. 

\bibitem{BoEr} P. Borwein, T. Erd\'elyi, \emph{Polynomials and polynomial inequalities}.
Graduate Texts in Mathematics, 161. Springer-Verlag, New York, 1995

\end{thebibliography}
\end{document}